\documentclass[reqno]{amsart}

\pdfoutput=1
\usepackage[all,cmtip,line]{xy}
\usepackage{tikz}

\usepackage{amsmath, amstext, amssymb, amsthm, amscd}
\usepackage{scalefnt}
\usepackage{microtype}

\usepackage[colorlinks,bookmarks=false]{hyperref}

\usepackage{url}

\hypersetup{citecolor=blue}

\newtheorem{theorem}{Theorem}

\newtheorem{lemma}[theorem]{Lemma}

\theoremstyle{definition}
\newtheorem{definition}[theorem]{Definition}
\newtheorem{remark}[theorem]{Remark}
\newtheorem{example}[theorem]{Example}

\newcommand{\eps}{\epsilon}
\newcommand{\la}{\lambda}
\newcommand{\G}{{\mathbb G}}
\newcommand{\R}{{\mathbb R}}

\newcommand{\Z}{{\mathbb Z}}
\newcommand{\sC}{{\mathcal C}}
\newcommand{\Q}{{\mathbb Q}}

\newcommand{\A}{{\mathbb A}}
\newcommand{\sE}{{\mathcal E}}
\newcommand{\Proj}{{\mathbb P}}
\newcommand{\Aut}{\mathrm{Aut}}
\newcommand{\Gal}{\mathrm{Gal}}
\newcommand{\I}{\mathrm{I}}

\newcommand{\III}{\mathrm{III}}
\newcommand{\II}{\mathrm{II}}
\newcommand{\Frob}{\mathrm{Frob}}

\DeclareMathOperator{\Pic}{Pic}
\DeclareMathOperator{\NS}{NS}
\DeclareMathOperator{\MW}{MW}

\usepackage{hyperref}

\title[Multiplicative excellent families for $E_7$ or
  $E_8$]{Multiplicative excellent families of elliptic surfaces of
  type $E_7$ or $E_8$}


\author{Abhinav Kumar}
\address{Department of Mathematics\\
Massachusetts Institute of Technology\\
Cambridge, MA 02139, USA}
\email{abhinav@math.mit.edu}

\author{Tetsuji Shioda}
\address{Department of Mathematics\\
Rikkyo University\\
Tokyo 171-8501, Japan}
\email{shioda@rikkyo.ac.jp}

\thanks{Kumar was supported in part by NSF CAREER grant DMS-0952486,
  and by a grant from the Solomon Buchsbaum Research Fund. Shioda was
  partially supported by JSPS Grant-in-Aid for Scientific Research
  (C)20540051.}

\date{January 26, 2015}

\begin{document}

\begin{abstract}

  We describe explicit multiplicative excellent families of rational
  elliptic surfaces with Galois group isomorphic to the Weyl group of
  the root lattices $E_7$ or $E_8$. The Weierstrass coefficients of
  each family are related by an invertible polynomial transformation
  to the generators of the multiplicative invariant ring of the
  associated Weyl group, given by the fundamental characters of the
  corresponding Lie group. As an application, we give examples of
  elliptic surfaces with multiplicative reduction and all sections
  defined over $\Q$ for most of the entries of fiber configurations
  and Mordell-Weil lattices in \cite{OS}, as well as examples of
  explicit polynomials with Galois group $W(E_7)$ or $W(E_8)$.
\end{abstract}

\maketitle

\section{Introduction}

Given an elliptic curve $E$ over a field $K$, the determination of its
Mordell-Weil group is a fundamental problem in algebraic geometry and
number theory. When $K = k(t)$ is a rational function field in one
variable, this question becomes a geometrical question of
understanding sections of an elliptic surface with section. Lattice
theoretic methods to attack this problem were described in
\cite{Sh1}. In particular, when $\sE \rightarrow \Proj^1_t$ is a
rational elliptic surface given as a minimal proper model of
$$
y^2 + a_1(t) xy + a_3(t) y = x^3 + a_2(t) x^2 + a_4(t) x + a_6(t)
$$
with $a_i(t) \in k[t]$ of degree at most $i$, the possible
configurations (types) of bad fibers and Mordell-Weil groups were
analyzed by Oguiso and Shioda \cite{OS}.

In \cite{Sh2}, the second author studied sections for some families of
elliptic surfaces with an additive fiber, by means of the
specialization map, and obtained a relation between the coefficients
of the Weierstrass equation and the fundamental invariants of the
corresponding Weyl groups. This was expanded in \cite{SU}, which
studied families with a bad fiber of additive reduction more
exhaustively. The formal notion of an excellent family was defined
(see the next section), and the authors found excellent families for
many of the ``admissible'' types.

The analysis of rational elliptic surfaces of high Mordell-Weil rank,
but with a fiber of multiplicative reduction, is much more
challenging. However, understanding this situation is arguably more
fundamental, since if we write down a ``random'' elliptic surface,
then with probability close to $1$ it will have Mordell-Weil lattice
$E_8$ and twelve nodal fibers (i.e. of multiplicative reduction). To
be more precise, if we choose Weierstrass coefficients $a_i(t)$ of
degree $i$, with coefficients chosen uniformly at random from among
rational numbers (say) of height at most $N$, then as $N \rightarrow
\infty$ the surface will satisfy the above condition with probability
approaching $1$. One can make a similar statement for rational
elliptic surfaces chosen to have Mordell-Weil lattice $E_7^*$,
$E_6^*$, etc.

In a recent work \cite{Sh4}, this study was carried out for elliptic
surfaces with a fiber of type $\I_3$ and Mordell-Weil lattice
isometric to $E_6^*$, through a ``multiplicative excellent family'' of
type $E_6$. We will describe this case briefly in Section
\ref{E6recap}. The main result of this paper shows that two explicitly
described families of rational elliptic surfaces with Mordell-Weil
lattices $E_7^*$ or $E_8$ are multiplicative excellent. The proof
involves a surprising connection with representation theory of the
corresponding Lie groups, and in particular, their fundamental
characters.  In particular, we deduce that the Weierstrass
coefficients give another natural set of generators for the
multiplicative invariants of the respective Weyl groups, as a
polynomial ring. We note that similar formulas were derived by Eguchi
and Sakai \cite{ES} using calculations from string theory and mirror
symmetry.

The idea of an excellent family is quite useful and important in
number theory. An excellent family of algebraic varieties leads to a
Galois extension $F(\mu)/F(\lambda)$ of two purely transcendental
extensions of a number field $F$ (say $\Q$), with Galois group a
desired finite group $G$. This setup has an immediate number theoretic
application, since one may specialize the parameters $\lambda$ and
apply Hilbert's irreducibility theorem to obtain Galois extensions
over $\Q$ with the same Galois group. Furthermore, we can make the
construction effective if appropriate properties of the group $G$ are
known (see examples \ref{bigexampleE7} and \ref{bigexampleE8} for the
case $G=W(E_7)$ or $W(E_8)$). At the same time, an excellent family
will give rise to a split situation very easily, by specializing the
parameters $\mu$ instead. For examples, in the situation considered in
our paper, we obtain elliptic curves over $\Q(t)$ with Mordell-Weil
rank $7$ or $8$ together with explicit generators for the Mordell-Weil
group (see examples \ref{splitexampleE7} and
\ref{splitexampleE8}). There are also applications to geometric
specialization or degeneration of the family. Therefore, it is
desirable (but quite nontrivial) to construct explicit excellent
families of algebraic varieties. Such a situation is quite rare in
general: theoretically, any finite reflection group is a candidate,
but it is not generally neatly related to an algebraic geometric
family. Hilbert treated the case of the symmetric group $S_n$,
corresponding to families of zero-dimensional varieties. Not many
examples were known before the (additive) excellent families for the
Weyl groups of the exceptional Lie groups $E_6$, $E_7$ and $E_8$ were
given in \cite{Sh2}, using the theory of Mordell-Weil lattices. Here,
we finish the story for the multiplicative excellent families for
these Weyl groups.

\section{Mordell-Weil lattices and excellent families}

Let $X \stackrel{\pi}{\rightarrow} \Proj^1$ be an elliptic surface
with section $\sigma: \Proj^1 \rightarrow X$, i.e. a proper relatively
minimal model of its generic fiber, which is an elliptic curve. We
denote the image of $\sigma$ by $O$, which we take to be the zero
section of the N\'eron model. We let $F$ be the class of a fiber in
$\Pic(X) \cong \NS(X)$, and let the reducible fibers of $\pi$ lie over
$\nu_1, \dots, \nu_k \in \Proj^1$. The non-identity components of
$\pi^{-1}(\nu_i)$ give rise to a sublattice $T_i$ of $\NS(X)$, which
is (the negative of) a root lattice (see \cite{Ko,T}). The {\em trivial
  lattice} $T$ is $\Z O \oplus \Z F \oplus (\oplus T_i)$, and we have
the isomorphism $\MW(X / \Proj^1) \cong \NS(X)/T$, which describes the
Mordell-Weil group. In fact, one can induce a positive definite
pairing on the Mordell-Weil group modulo torsion, by inducing it from
the negative of the intersection pairing on $\NS(X)$. We refer the
reader to \cite{Sh1} for more details. In this paper, we will call
$\oplus T_i$ the {\em fibral lattice}.

Next, we recall the notion of an {\em excellent family} with Galois
group $G$ from \cite{SU}. Suppose $X \rightarrow \A^n$ is a family of
algebraic varieties, varying with respect to $n$ parameters $\la_1,
\dots, \la_n$. The generic member of this family $X_\la$ is a variety
over the rational function field $k_0 = \Q(\la)$. Let $k =
\overline{k_0}$ be the algebraic closure, and suppose that
$\sC(X_\la)$ is a group of algebraic cycles on $X_\la$ over the field
$k$ (in other words, it is a group of algebraic cycles on $X_\la
\times_{k_0} k$). Suppose in addition that there is an isomorphism
$\phi_\la: \sC(X_\la) \otimes \Q \cong V$ for a fixed vector space
$V$, and $\sC(X_\la)$ is preserved by the Galois group
$\Gal(k/k_0)$. Then we have the Galois representation
$$
\rho_\la: \Gal(k/k_0) \rightarrow \Aut(\sC(X_\la)) \rightarrow \Aut(V).
$$
We let $k_\la$ be the fixed field of the kernel of $\rho_\la$, i.e. it
is the smallest extension of $k_0$ over which the cycles of $\sC(\la)$
are defined. We call it the {\em splitting field} of $\sC(X_\la)$.

Now let $G$ be a finite reflection group acting on the space $V$. 
\begin{definition} \label{adddefn}
  We say $\{ X_\la \}$ is an {\em excellent family} with Galois group
  $G$ if the following conditions hold:
\begin{enumerate}
\item The image of $\rho_\la$ is equal to $G$.
\item There is a $\Gal(k/k_0)$-equivariant evaluation map $s:
  \sC(X_\la) \rightarrow k$.
\item There exists a basis $\{Z_1, \dots, Z_n \}$ of $\sC(X_\la)$ such
  that if we set $u_i = s(Z_i)$, then $u_1, \dots, u_n$ are
  algebraically independent over $\Q$.
\item $\Q[u_1, \dots, u_n]^G = \Q[\la_1, \dots, \la_n]$.
\end{enumerate}
\end{definition}

As an example, for $G = W(E_8)$, consider the following family of
rational elliptic surfaces
$$
y^2 = x^3 + x \left( \sum_{i=0}^3 p_{20 - 6i} t^i \right) + \left( \sum_{j=0}^3 p_{30 - 6j} t^j + t^5 \right)
$$
over $k_0 = \Q(\la)$, with $\la = (p_2, p_8, p_{12}, p_{14}, p_{18}, p_{20},
p_{24}, p_{30})$. It is shown in \cite{Sh2} that this is an excellent
family with Galois group $G$. The $p_i$ are related to the fundamental
invariants of the Weyl group of $E_8$, as is suggested by their
degrees (subscripts).

We now define the notion of a {\em multiplicative excellent family}
for a group $G$. As before, $X \rightarrow \A^n$ is a family of
algebraic varieties, varying with respect to $\la = (\la_1, \dots,
\la_n)$, and $\sC(X_\la)$ is a group of algebraic cycles on $X_\la$,
isomorphic (via a fixed isomorphism) to a fixed abelian group $M$. The
fields $k_0$ and $k$ are as before, and we have a Galois
representation
$$
\rho_\la: \Gal(k/k_0) \rightarrow \Aut(\sC(X_\la)) \rightarrow \Aut(M).
$$

Suppose that $G$ is a group acting on $M$.

\begin{definition} \label{multdefn}
  We say $\{ X_\la \}$ is a {\em multiplicative excellent family} with
  Galois group $G$ if the following conditions hold:
\begin{enumerate}
\item The image of $\rho_\la$ is equal to $G$. 
\item There is a $\Gal(k/k_0)$-equivariant evaluation map $s:
  \sC(X_\la) \rightarrow k^*$.
\item There exists a basis $\{Z_1, \dots, Z_n \}$ of $\sC(X_\la)$ such
  that if we set $u_i = s(Z_i)$, then $u_1, \dots, u_n$ are
  algebraically independent over $\Q$.
\item $\Q[u_1, \dots, u_n, u_1^{-1}, \dots, u_n^{-1}]^G = \Q[\la_1, \dots, \la_n]$.
\end{enumerate}
\end{definition}

\begin{remark} 
Though we use similar notation, the specialization map $s$ and
the $u_i$ in the multiplicative case are quite different from the ones
in the additive case. Intuitively, one may think of these as
exponentiated versions of the corresponding objects in the additive
case. However, we want the specialization map to be an algebraic
morphism, and so in general (additive) excellent families specified by
Definition \ref{adddefn} will be very different from multiplicative
excellent families specified by Definition \ref{multdefn}.
\end{remark}

In our examples, $G$ will be a finite reflection group acting on a
lattice in Euclidean space, which will be our choice for $M$. However,
what is relevant here is not the ring of (additive) invariants of $G$
on the vector space spanned by $M$. Instead, note that the action of
$G$ on $M$ gives rise to a ``multiplicative'' or ``monomial'' action
of $G$ on the group algebra $\Q[M]$, and we will be interested in the
polynomials on this space which are invariant under $G$. This is the
subject of {\em multiplicative invariant theory} (see, for example,
\cite{L}). In the case when $G$ is the automorphism group of a root
lattice or root system, multiplicative invariants were classically
studied by using the terminology of ``exponentiated'' roots
$e^{\alpha}$ (for instance, see \cite[Section VI.3]{B}).

\section{The $E_6$ case} \label{E6recap}

We now briefly describe the construction of multiplicative excellent
family in \cite{Sh4}. Consider the family of rational elliptic surfaces
$S_\la$ with Weierstrass equation
$$
y^2 + txy = x^3 + (p_0 + p_1 t + p_2 t^2) \,x + q_0 + q_1 t + q_2 t^2 + t^3
$$
with parameter $\la = (p_0, p_1, p_2, q_0, q_1, q_2)$. The surface
$S_\la$ generically only has one reducible fiber at $t = \infty$, of
type $\I_3$. Therefore, the Mordell-Weil lattice $M_\la$ of $S_\la$ is
isomorphic to $E_6^*$. There are $54$ minimal sections of height
$4/3$, and exactly half of them have the property that $x$ and $y$ are
linear in $t$. If we have
\begin{align*}
x &= at + b \\
y &= ct + d,
\end{align*}
then substituting these back in to the Weierstrass equation, we get a
system of equations, and we may easily eliminate $b,c,d$ from the
system to get a monic equation of degree $27$ (subject to a genericity
assumption), which we write as $\Phi_\la(a) = 0$. Also, note that the
specialization of a section of height $4/3$ to the fiber at $\infty$
gives us a point on one of the two non-identity components of the
special fiber of the N\'eron model (the same component for all the
$27$ sections). Identifying the smooth points of this component with
$\G_m \times \{1\} \subset \G_m \times (\Z/3\Z)$, the specialization
map $s$ takes the section to $(-1/a, 1)$. Let the specializations be
$s_i = -1/a_i$ for $1 \leq i \leq 27$. We have
\begin{align*}
\Phi_\la(X) &= \prod_{i=1}^{27} (X - a_i) \\
        &= \prod_{i=1}^{27} (X + 1/s_i) \\
        &= X^{27} + \eps_{-1} X^{26} + \eps_{-2} X^{25} + \dots + \eps_4 X^4 + \eps_3 X^3 + \eps_2 X^2 + \eps_1 X + 1.
\end{align*}
Here $\eps_i$ is the $i$'th elementary symmetric polynomial of the
$s_i$ and $\eps_{-i}$ that of the $1/s_i$. The coefficients of
$\Phi_\la(X)$ are polynomials in the coordinates of $\la$, and we may
use the equations for $\eps_1, \eps_2, \eps_3, \eps_4, \eps_{-1}$ and
$\eps_{-2}$ to solve for $p_0, \dots, q_3$. However, the resulting
solution has $\eps_{-2}$ in the denominator. We may remedy this situation
as follows: consider the construction of $E_6^*$ as described in
\cite{Sh5}: let $v_1, \dots, v_6$ be vectors in $\R^6$ with $\langle
v_i, v_j \rangle = \delta_{ij} + 1/3$, and let $u = (\sum v_i)/3$. The
$\Z$-span of $v_1, \dots, v_6, u$ is a lattice $L$ isometric to
$E_6^*$. It is clear that $v_1, \dots,v_5, u$ form a basis of
$L$. Here, we choose an isometry between the Mordell-Weil lattice and
the lattice $L$, and let the specializations of $v_1, \dots, v_6, u$
be $s_1, \dots, s_6, r$ respectively. These satisfy $s_1 s_2 \dots s_6
= r^3$. The fifty-four nonzero minimal vectors of $E_6^*$ split up
into two cosets (modulo $E_6$) of twenty-seven each, of which we have
chosen one. The specializations of these twenty-seven special sections
are given by
$$
\{ s_1, \dots, s_{27} \} := \{ s_i : 1 \leq i \leq 6 \} \bigcup \{ s_i/r : 1 \leq i \leq 6 \} \bigcup \{ r/(s_i s_j) : 1 \leq i < j \leq 6 \}.
$$
Let 
$$
\delta_1 = r + \frac{1}{r} + \sum_{ i \neq j} \frac{s_i}{s_j} + \sum_{i < j < k} \left( \frac{r}{s_i s_j s_k} + \frac{s_i s_j s_k}{r} \right).
$$
Then $\delta_1$ belongs to the $G = W(E_6)$-invariants of $\Q[s_1,
\dots, s_5, r, s_1^{-1}, \dots, s_5^{-1},r^{-1}]$, and it is shown by explicit
computation in \cite{Sh4} that
$$
\Q[s_1, \dots, s_5, r, s_1^{-1}, \dots, s_5^{-1}, r^{-1}]^G = \Q[\delta_1, \eps_1, \eps_2, \eps_3, \eps_{-1}, \eps_{-2}] = \Q[p_0, p_1, p_2, q_0, q_1, q_2].
$$
The explicit relation showing the second equality is as follows.
\begin{align*}
& \delta_1 = -2 p_1 \\
& \eps_1 = 6 p_2 \\
& \eps_{-1} = p_2^2 - q_2 \\
& \eps_2 = 13 p_2^2 + p_0 - q_2 \\
& \eps_{-2} = -2p_1 p_2 + 6p_2 + q_1 \\
& \eps_3 = 8 p_2^3 + 2 p_0 p_2 + p_1^2 - 6 p_1 - q_0 + 9.
\end{align*}
We make the following additional observation. The six fundamental
representations of the Lie algebra $E_6$ correspond to the fundamental
weights in the following diagram, which displays the standard labeling
of these representations.

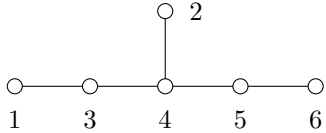
\begin{center}
\begin{tikzpicture}

\draw (0,0)--(4,0);
\draw (2,0)--(2,1);

\fill [white] (0,0) circle (0.1);
\fill [white] (1,0) circle (0.1);
\fill [white] (2,0) circle (0.1);
\fill [white] (3,0) circle (0.1);
\fill [white] (4,0) circle (0.1);
\fill [white] (2,1) circle (0.1);

\draw (0,0) circle (0.1);
\draw (1,0) circle (0.1);
\draw (2,0) circle (0.1);
\draw (3,0) circle (0.1);
\draw (4,0) circle (0.1);
\draw (2,1) circle (0.1);

\draw (0,-0.2) [below] node{$1$};
\draw (1,-0.2) [below] node{$3$};
\draw (2,-0.2) [below] node{$4$};
\draw (3,-0.2) [below] node{$5$};
\draw (4,-0.2) [below] node{$6$};
\draw (2.2,1) [right] node{$2$};

\end{tikzpicture}
\end{center}
The dimensions of these representations $V_1, \dots, V_6$ are $27, 78,
351, 2925, 351, 27$ respectively, and their characters are related to
$\eps_1, \eps_2, \eps_3, \eps_{-1}, \eps_{-2}, \delta_1$ by the
following nice transformation.
\begin{align*}
\chi_1 &= \eps_1 & 
\chi_2 &= \delta_1 + 6 &
\chi_3 &= \eps_2 \\
\chi_4 &= \eps_3 &
\chi_5 &= \eps_{-2} &
\chi_6 &= \eps_{-1}.
\end{align*}

This explains the reason for bringing in $\delta_1$ into the picture,
and also why there is a denominator when solving for $p_0, \dots, q_2$
in terms of $\eps_1, \dots, \eps_4, \eps_{-1},\eps_{-2}$, as remarked
in \cite{Sh4}. The coefficients $\eps_j$ are essentially the characters
of $\bigwedge^j V$, where $V = V_1$ is the first fundamental representation,
while $\eps_{-j}$ are those of $\bigwedge^j V^*$, where $V_6 = V^*$. Note
that $\bigwedge^3 V \cong \bigwedge^3 V^*$. Therefore, from the expressions for
$\eps_1, \eps_2, \eps_3, \eps_{-1}, \eps_{-2}$, we may obtain $p_2,
q_2, p_0, q_1, q_0$, in terms of the remaining variable $p_1$, without
introducing any denominators. However, representation $V_2$ cannot be
obtained as a direct summand with multiplicity $1$ from a tensor
product of $\bigwedge^j V$ (for $1 \leq j \leq 3$) and $\bigwedge^k V^*$ (for $1
\leq k \leq 2$). On the other hand, we do have the following
isomorphism:
$$
(V_2 \otimes V_5) \oplus V_5 \oplus V_1 \cong \bigwedge^4 V_1 \oplus (V_3 \otimes V_6) \oplus (V_6 \otimes V_6).
$$
Therefore, we are able to solve for $p_1$ if we introduce a
denominator of $\eps_{-2}$, which is the character of $V_5$.

\section{The $E_7$ case}

\subsection{Results}

Next, we exhibit a multiplicative excellent family for the Weyl group
of $E_7$. It is given by the Weierstrass equation
$$
y^2 + txy = x^3 + (p_0 + p_1 t +p_2 t^2) \, x + q_0 + q_1 t + q_2 t^2 + q_3 t^3 - t^4.
$$ For generic $\la = (p_0, \dots, p_2, q_0, \dots, q_3)$, this
rational elliptic surface $X_\la$ has a fiber of type $\I_2$ at $t =
\infty$, and no other reducible fibers. Hence, the Mordell-Weil group
$M_\la$ is $E_7^*$. We note that any elliptic surface with a fiber of
type $\I_2$ can be put into this Weierstrass form (in general over a
small degree algebraic extension of the ground field), after a
fractional linear transformation of the parameter $t$, and Weierstrass
transformations of $x,y$.

\begin{lemma}
  The smooth part of the special fiber is isomorphic to the group
  scheme $\G_m \times \Z/2\Z$. The identity component is the
  non-singular part of the curve $y^2 + xy = x^3$. The $x$- and
  $y$-coordinates of a section of height $2$ are polynomials of
  degrees $2$ and $3$ respectively, and its specialization at $t =
  \infty$ is $(\lim_{t \to \infty} (y + tx)/y, 0) \in k^* \times
  \{0,1\}$. A section of height $3/2$ has $x$ and $y$ coordinates of
  the form
\begin{align*}
x &= at + b \\
y &= ct^2 + dt + e.
\end{align*}
and specializes at $t = \infty$ to $(c,1)$.
\end{lemma}
\begin{proof}
  First, observe that to get a local chart for the elliptic surface
  near $t = \infty$, we set $x = \widetilde{x}/u^2$, $y =
  \widetilde{y}/u^3$ and $t = 1/u$, and look for $u$ near
  $0$. Therefore, the special fiber (before blow-up) is given by
  $\bar{y}^2 + \bar{x} \bar{y} = \bar{x}^3$, where $\bar{x} =
  \widetilde{x} |_{u = 0}$ and $\bar{y} = \widetilde{y}|_{u = 0}$ are
  the reductions of the coordinates at $u = 0$ respectively. It is an
  easy exercise to parametrize the smooth locus of this curve: it is
  given, for instance, by $\bar{x} = s/(s-1)^2, \bar{y} =
  s/(s-1)^3$. We then check that $s = (\bar{y}+\bar{x})/\bar{y}$ and
  the map from the smooth locus to $\G_m$ which takes the point
  $(\bar{x}, \bar{y})$ to $s$ is a homomorphism from the secant group
  law to multiplication in $k^*$. This proves the first half of the
  lemma. Note that we could just as well have taken $1/s$ to be the
  parameter on $\G_m$; our choice is a matter of convention. To prove
  the specialization law for sections of height $3/2$, we may, for
  instance, take the sum of such a section $Q$ with a section $P$ of
  height $2$ with specialization $(s,0)$. A direct calculation shows
  that the $y$-coordinate of the sum has top (quadratic) coefficient
  $cs$. Therefore the specialization of $Q$ must have the form $\kappa
  c$, where $\kappa$ is a constant not depending on $Q$. Finally, the
  sum of two sections $Q_1$ and $Q_2$ of height $3/2$ and having
  coefficients $c_1$ and $c_2$ for the $t^2$ term of their
  $y$-coordinates can be checked to specialize to $(c_1 c_2, 0)$. It
  follows that $\kappa = \pm 1$, and we take the plus sign as a
  convention. (It is easy to see that both choices of sign are
  legitimate, since the sections of height $2$ generate a copy of
  $E_7$, whereas the sections of height $3/2$ lie in the nontrivial
  coset of $E_7$ in $E_7^*$).
\end{proof}

There are $56$ sections of height $3/2$, with $x$ and $y$ coordinates
in the form above. Substituting the above formulas for $x$ and $y$
into the Weierstrass equation, we get the following system of
equations.
\begin{align*}
c^2 + ac + 1  &= 0 \\
q_3 + a p_2 + a^3 &= (2c + a) d + bc \\
q_2 + b p_2 + 3 a^2 b &= (2c + a) e + (b+d) d \\
q_1 + b p_1 + a p_0  + 3 a b^2 &= (2d + b) e \\
q_0 + b p_0 + b^3 &= e^2.
\end{align*}

We solve for $a$ and $b$ from the first and second equations, and then
$e$ from the third, assuming $c \neq 1$. Substituting these values
back into the last two equations, we get two equations in the
variables $c$ and $d$. Taking the resultant of these two equations
with respect to $d$, and dividing by $c^{30} (c^2 - 1)^4$, we obtain
an equation of degree $56$ in $c$, which is monic, reciprocal and has
coefficients in $\Z[\la] = \Z[p_0, \dots, q_3]$. We denote this
polynomial by
$$
\Phi_\la(X) = \prod_{i=1}^{56} \left(X - s(P) \right) = X^{56} + \eps_1 X^{55} + \eps_2 X^{54}+  \dots + \eps_1 X + \eps_0,
$$
where $P$ ranges over the $56$ minimal sections of height $3/2$.  It
is clear that $a,b,d,e$ are rational functions of $c$ with
coefficients in $k_0$.

We have a Galois representation on the Mordell-Weil lattice
$$
\rho_\la: \Gal(k/k_0) \rightarrow \Aut(M_\la) \cong \Aut(E_7^*).
$$
Here $\Aut(E_7^*) \cong \Aut(E_7) \cong W(E_7)$, the Weyl group of
type $E_7$. The {\em splitting field} of $M_\la$ is the fixed field
$k_\la$ of $\mathrm{Ker}(\rho_\la)$. By definition, $\Gal(k_\la/k_0)
\cong \textrm{Im}(\rho_\la)$. The splitting field $k_\la$ is equal to
the splitting field of the polynomial $\Phi_\la(X)$ over $k_0$, since
the Mordell-Weil group is generated by the $56$ sections of smallest
height $P_i = (a_i t + b_i, c_i t^2 + d_i t + e_i)$. We also have that
$$
k_\la = k_0(P_1, \dots, P_{56}) = k_0(c_1, \dots, c_{56}).
$$ We shall sometimes write $e^{\alpha}$, (for $\alpha \in E_7^*$ a
minimal vector) to refer to the specializations of these sections
$c(P_i)$, for convenience.

\begin{theorem} \label{excellentE7}
  Assume that $\lambda$ is generic over $\Q$, i.e the coordinates
  $p_0, \dots, q_3$ are algebraically independent over $\Q$. Then
\begin{enumerate}
\item $\rho_\la$ induces an isomorphism $\Gal(k_\la / k_0) \cong W(E_7)$.
\item The splitting field $k_\la$ is a purely transcendental extension
  of $\Q$, isomorphic to the function field $\Q(Y)$ of the toric
  hypersurface $Y \subset \G_m^8$ defined by $s_1 \dots s_7 =
  r^3$. There is an action of $W(E_7)$ on $Y$ such that
  $\Q(Y)^{W(E_7)} = k_\la^{W(E_7)} = k_0$.
\item The ring of $W(E_7)$-invariants in the affine coordinate ring
 $$
  \Q[Y] = \frac{\Q[s_i, r, 1/s_i, 1/r]}{(s_1\dots s_7 - r^3)} \cong \Q[s_1, \dots, s_6, r, s_1^{-1}, \dots, s_6^{-1}, r^{-1}]
 $$ 
 is equal to the polynomial ring
  $\Q[\la]$:
$$
\Q[Y]^{W(E_7)} = \Q[\la] = \Q[p_0, p_1, p_2, q_0, q_1, q_2, q_3].
$$
\end{enumerate}
\end{theorem}

In fact, we shall prove an explicit, invertible polynomial relation
between the Weierstrass coefficients $\lambda$ and the fundamental
characters for $E_7$. Let $V_1, \dots, V_7$ be the fundamental
representations of $E_7$, and $\chi_1, \dots, \chi_7$ their characters
(on a maximal torus), as labelled below. For a description of the
fundamental modules for the exceptional Lie groups see \cite[Section
  13.8]{C}.  

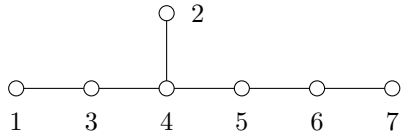
\begin{center}
\begin{tikzpicture}

\draw (0,0)--(5,0);
\draw (2,0)--(2,1);

\fill [white] (0,0) circle (0.1);
\fill [white] (1,0) circle (0.1);
\fill [white] (2,0) circle (0.1);
\fill [white] (3,0) circle (0.1);
\fill [white] (4,0) circle (0.1);
\fill [white] (5,0) circle (0.1);
\fill [white] (2,1) circle (0.1);

\draw (0,0) circle (0.1);
\draw (1,0) circle (0.1);
\draw (2,0) circle (0.1);
\draw (3,0) circle (0.1);
\draw (4,0) circle (0.1);
\draw (5,0) circle (0.1);
\draw (2,1) circle (0.1);

\draw (0,-0.2) [below] node{$1$};
\draw (1,-0.2) [below] node{$3$};
\draw (2,-0.2) [below] node{$4$};
\draw (3,-0.2) [below] node{$5$};
\draw (4,-0.2) [below] node{$6$};
\draw (5,-0.2) [below] node{$7$};
\draw (2.2,1) [right] node{$2$};

\end{tikzpicture}
\end{center}

Note that since the weight lattice $E_7^*$ has been equipped with a
nice set of generators $(v_1, \dots, v_7, u)$ with $\sum v_i = 3 u$
(as in \cite{Sh5}), the characters $\chi_1, \dots, \chi_7$ lie in the
ring of Laurent polynomials $\Q[s_i, r, 1/s_i, 1/r]$ where $s_i$
corresponds to $e^{v_i}$ and $r$ to $e^u$, and are obviously invariant
under the (multiplicative) action of the Weyl group on this ring of
Laurent polynomials. Explicit formulae for the $\chi_i$ are given in
the auxiliary files.

We also note that the roots of $\Phi_\la$ are given by 
$$
s_i, \frac{1}{s_i} \textrm{ for } 1\leq i \leq 7 \quad \textrm{ and } \quad \frac{r}{s_i s_j}, \frac{s_i s_j}{r} \textrm{ for } 1 \leq i < j \leq 7.
$$

\begin{theorem} \label{explicitE7}
For generic $\lambda$ over $\Q$, we have 
$$
\Q[\chi_1, \dots, \chi_7] = \Q[p_0, p_1, p_2, q_0, q_1, q_2, q_3].
$$
The transformation between these sets of generators is
\begin{flalign*}
\chi_1 &= 6 p_2+25 \\
\chi_2 &= 6 q_3-2 p_1 \\
\chi_3 &= -q_2+13 p_2^2+108 p_2+p_0+221 \\
\chi_4 &= 9 q_3^2-6 p_1 q_3-q_2-q_0+8 p_2^3+85 p_2^2+(2 p_0+300) p_2+p_1^2+10 p_0+350 \\
\chi_5 &= (6 p_2+26) q_3+q_1-2 p_1 p_2-10 p_1 \\
\chi_6 &= -q_2+p_2^2+12 p_2+27 \\
\chi_7 &= q_3 \\
& \\
\textrm{with inverse}\\
& \\
p_2 &= (\chi_1-25)/6 \\
p_1 &= (6 \chi_7-\chi_2)/2 \\
p_0 &= -(3 \chi_6-3 \chi_3+\chi_1^2-2 \chi_1+7)/3 \\
q_3 &= \chi_7 \\
q_2 &= -(36 \chi_6-\chi_1^2-22 \chi_1+203)/36 \\
q_1 &= (24 \chi_7+6 \chi_5+(-\chi_1-5) \chi_2)/6 \\
q_0 &= (27 \chi_2^2-8 \chi_1^3-84 \chi_1^2+120 \chi_1-136)/108  \\
    & \quad -(\chi_1+2) \chi_6/3 - \chi_4 + (\chi_1+ 5) \chi_3/3.
\end{flalign*}
\end{theorem}

We note that our formulas agree with those of Eguchi and Sakai
\cite{ES}, who seem to derive these by using an ansatz.

Next, we describe two examples through specialization, one of ``small
Galois'' (in which all sections are defined over $\Q[t]$) and one with
``big Galois'' (which has Galois group the full Weyl group).

\begin{example} \label{splitexampleE7}
The values 
\begin{flalign*}
p_0 &= 244655370905444111/(3 \mu^2),  \quad p_1 = -4788369529481641525125/(16 \mu^2) \\
q_3 &= 184185687325/(4 \mu), \quad p_2 = 199937106590279644475038924955076599/(12 \mu^4) \\
q_2 &= 57918534120411335989995011407800421/(9 \mu^3) \\
q_1 &= -179880916617213624948875556502808560625/(4 \mu^4) \\
q_0 &= 35316143754919755115952802080469762936626890880469201091/(1728 \mu^6) ,
\end{flalign*}
where $\mu = 2 \cdot 3 \cdot 5 \cdot 7 \cdot 11 \cdot 13 \cdot 17 =
102102$, give rise to an elliptic surface for which we have $r = 2,
s_1 = 3, s_2 = 5, s_3 = 7, s_4 = 11, s_5 = 13, s_6 = 17$, the simplest
choice of multiplicatively independent elements. The Mordell-Weil
group has a basis of sections for which $c \in
\{3,5,7,11,13,17,15/2\}$. We write down their $x$-coordinates below:
\begin{align*}
x(P_1) &= -(10/3)t - 707606695171055129/1563722760600 \\ 
x(P_2) &= -(26/5)t - 611410735289928023/1563722760600 \\ 
x(P_3) &= -(50/7)t - 513728975686763429/1563722760600 \\ 
x(P_4) &= -(122/11)t - 316023939417997169/1563722760600 \\ 
x(P_5) &= -(170/13)t - 216677827127591279/1563722760600 \\ 
x(P_6) &= -(290/17)t - 17562556436754779/1563722760600 \\ 
x(P_7) &= -(229/30)t - 140574879644393807/390930690150.
\end{align*}
In the auxiliary files the $x$-and $y$-coordinates are listed, and it is verified that they satisfy the Weierstrass equation.
\end{example}

\begin{example} \label{bigexampleE7}
The value $\la = \la_0 := (1,1,1,1,1,1,1)$ gives rise to an explicit
polynomial $f(X) = \Phi_{\la_0}(X)$, given by 
\begin{align*}
f(X) &= X^{56} - X^{55} + 40X^{54} - 22X^{53} + 797X^{52} - 190X^{51} + 9878X^{50} - 1513X^{49} \\
& \quad +  82195X^{48} - 17689X^{47} + 496844X^{46} - 175584X^{45} + 2336237X^{44} \\
& \quad - 1196652X^{43} + 8957717X^{42} - 5726683X^{41} + 28574146X^{40}  \\
& \quad - 20119954X^{39} +  75465618X^{38} - 53541106X^{37} + 163074206X^{36}  \\
& \quad - 110505921X^{35} +  287854250X^{34} - 181247607X^{33} + 420186200X^{32} \\
& \quad - 243591901X^{31} +  518626022X^{30} - 278343633X^{29} + 554315411X^{28}  \\
& \quad - 278343633X^{27} +  518626022X^{26} - 243591901X^{25} + 420186200X^{24}  \\
& \quad - 181247607X^{23} +  287854250X^{22} - 110505921X^{21} + 163074206X^{20} \\
& \quad - 53541106X^{19} +  75465618X^{18} - 20119954X^{17} + 28574146X^{16} \\
& \quad  - 5726683X^{15} + 8957717X^{14} - 1196652X^{13} + 2336237X^{12}  \\
& \quad - 175584X^{11} + 496844X^{10} - 17689X^9 +  82195X^8 - 1513X^7 \\
& \quad + 9878X^6 - 190X^5 + 797X^4 - 22X^3 + 40X^2 - X + 1,
\end{align*}
for which we can show that the Galois group is the full group
$W(E_7)$, as follows. The reduction of $f(X)$ modulo $7$ shows that
$\Frob_7$ has cycle decomposition of type $(7)^8$, and similarly,
$\Frob_{101}$ has cycle decomposition of type $(3)^2(5)^4(15)^2$. This
implies, as in \cite[Example 7.6]{Sh3}, that the Galois group is the
entire Weyl group.
\end{example}

We can also describe degenerations of this family of rational elliptic
surfaces $X_\la$ by the method of ``vanishing roots'', where we drop
the genericity assumption, and consider the situation where the
elliptic fibration might have additional reducible fibers. Let $\psi:
Y \rightarrow \A^7$ be the finite surjective morphism associated to
$$
\Q[p_0, \dots, q_3] \hookrightarrow \Q[Y] \cong \Q[s_1, \dots, s_6, r, s_1^{-1}, \dots, s_6^{-1}, r^{-1}].
$$ 
For $\xi = (s_1, \dots, s_7, r) \in Y$, let the multiset $\Pi_\xi$
consist of the $126$ elements $s_i/r$ and $r/s_i$ (for $1 \leq i \leq
7$), $s_i/s_j$ (for $1 \leq i \neq j \leq 7$) and $s_i s_j s_k/r$ and
$r/(s_i s_j s_k)$ for $1 \leq i < j < k \leq 7$, corresponding to the
$126$ roots of $E_7$. Let $2\nu(\xi)$ be the number of times $1$
appears in $\Pi_\xi$, which is also the multiplicity of $1$ as a root
of $\Psi_\la(X)$ (to be defined in Section \ref{E7proofs}), where
$\la = \psi(\xi)$. We call the associated roots of $E_7$ the {\em
  vanishing roots}, in analogy with vanishing cycles in the
deformation of singularities. By abuse of notation we label the
rational elliptic surface $X_\la$ as $X_\xi$.

\begin{theorem} \label{E7degen}
  The surface $X_\xi$ has new reducible fibers (necessarily at $t \neq
  \infty$) if and only if $\nu(\xi) > 0$. The number of roots in the
  root lattice $T_{\textrm{new}}$ is equal to $2\nu(\xi)$, where
  $T_{\textrm{new}} := \oplus_{v \neq \infty} T_v$ is the new part of
  the trivial lattice.
\end{theorem}

We may use this result to produce specializations with trivial lattice
including $A_1$, corresponding to the entries in the table of
\cite[Section 1]{OS}. Note that in earlier work \cite{Sh2, SU},
examples of rational elliptic surfaces were produced with a fiber of
additive type, for instance, a fiber of type $\III$ (which contributes
$A_1$ to the trivial lattice) or a fiber of type $\II$. Using our
excellent family, we can produce examples with the $A_1$ fiber being
of multiplicative type $\I_2$ and all other irreducible singular
fibers being nodal (that is, $\I_1$). We list below those types which
are not already covered by \cite{Sh4}. To produce these examples, we
use an embedding of the new part $T_{\textrm{new}}$ of the fibral
lattice into $E_7$, which gives us any extra conditions satisfied by
$s_1, \dots, s_7, r$.  The following multiplicative version of the
labeling of simple roots of $E_7$ is useful (compare \cite{Sh5}).

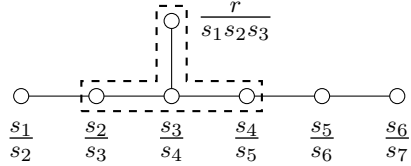
\begin{center}
{\scalefont{1.3}
\begin{tikzpicture}
\draw (0,0)--(5,0);
\draw (2,0)--(2,1);

\fill [white] (0,0) circle (0.1);
\fill [white] (1,0) circle (0.1);
\fill [white] (2,0) circle (0.1);
\fill [white] (3,0) circle (0.1);
\fill [white] (4,0) circle (0.1);
\fill [white] (5,0) circle (0.1);
\fill [white] (2,1) circle (0.1);

\draw (0,0) circle (0.1);
\draw (1,0) circle (0.1);
\draw (2,0) circle (0.1);
\draw (3,0) circle (0.1);
\draw (4,0) circle (0.1);
\draw (5,0) circle (0.1);
\draw (2,1) circle (0.1);

\draw (0,-0.2) [below] node{$\frac{s_1}{s_2}$};
\draw (1,-0.2) [below] node{$\frac{s_2}{s_3}$};
\draw (2,-0.2) [below] node{$\frac{s_3}{s_4}$};
\draw (3,-0.2) [below] node{$\frac{s_4}{s_5}$};
\draw (4,-0.2) [below] node{$\frac{s_5}{s_6}$};
\draw (5,-0.2) [below] node{$\frac{s_6}{s_7}$};
\draw (2.2,1) [right] node{$\frac{r}{s_1s_2s_3}$};

\draw [thick,dashed] (0.8, -0.2)--(3.2,-0.2)--(3.2,0.2)--(2.2,0.2)--(2.2,1.2)--(1.8,1.2)--(1.8,0.2)--(0.8,0.2)--(0.8,-0.2);

\end{tikzpicture}
}
\end{center}
For instance, to produce the example in line $18$ of the table
(i.e. with $T_{\textrm{new}} = D_4$), we may use the embedding into
$E_7$ indicated by embedding the $D_4$ Dynkin diagram within the
dashed lines in the figure above. Thus, we must force $s_2 = s_3 = s_4
= s_5$ and $r = s_1 s_2 s_3$, and a simple solution with no extra
coincidences is given in the rightmost column (note that $s_7 =
18^3/(2\cdot 3^4 \cdot 5) = 36/5$).

$$
\begin{array}{cccl}
\textrm{Type in \cite{OS}} & \textrm{Fibral lattice} & \textrm{MW group} & \{s_1, \dots, s_6, r \} \\
2 & A_1 & E_7^* & 3,5,7,11,13,17,2 \\
4 & A_1^2 & D_6^* & 3,3,5,7,11,13,2 \\
7 & A_1^3 & D_4^* \oplus A_1^*&  3,3,5,5,7,11,2 \\
10 & A_1 \oplus A_3 & A_1^* \oplus A_3^* & 3,3,3,3,5,7,2 \\
13 & A_1^4 & D_4^* \oplus \Z/2\Z & -1,2,3,5,7,49/30,7 \\
14 & A_1^4 & A_1^{*4}& 3,3,5,5,7,7,2 \\
17 & A_1 \oplus A_4 & \frac{1}{10} \footnotesize{ \left( \begin{array}{ccc} 3 & 1 & -1 \\ 1 & 7 & 3 \\ -1 & 3 & 7 \end{array}\right) } &  3,3,3,3,3,5,2 \\
18 & A_1 \oplus D_4 & A_1^{*3} &  2,3,3,3,3,5,18 \\
21 & A_1^2 \oplus A_3 & A_3^* \oplus \Z/2\Z  & 3,5,60,30,30,30,900 \\
22 & A_1^2 \oplus A_3 & A_1^{*2} \oplus \langle 1/4 \rangle &  3,3,5,5,5,5,2 \\
24 & A_1^5 & A_1^{*3} \oplus \Z/2\Z & 15/4,2,2,3,3,5,15 \\
28 & A_1 \oplus A_5 & A_2^* \oplus \Z/2\Z & 2,3,6,6,6,6,36 \\
29 & A_1 \oplus A_5 & A_1^* \oplus \langle 1/6 \rangle & 2,2,2,2,2,2,3 \\
30 & A_1 \oplus D_5 & A_1^* \oplus \langle 1/4 \rangle & 2,2,2,2,2,3,8 \\
33 & A_1^2 \oplus A_4 & \frac{1}{10} \footnotesize{ \left( \begin{array}{cc} 2 & 1 \\ 1 & 3 \end{array} \right) } & 2,2,3,3,3,3,12 \\
34 & A_1^2 \oplus D_4 & A_1^{*2} \oplus \Z/2\Z & 2,3,3,3,3,6,18 
\end{array} 
$$

$$
\begin{array}{cccl}
\textrm{Type in \cite{OS}} & \textrm{Fibral lattice} & \textrm{MW group} & \{s_1, \dots, s_6, r \} \\
38 & A_1^3 \oplus A_3 & A_1^* \oplus \langle 1/4 \rangle \oplus \Z/2\Z & 2,2,3,3,3,4,12 \\
42 & A_1^6 & A_1^{*2} \oplus (\Z/2\Z)^2 & 6,-1,-1,2,2,3,6 \\
47 & A_1 \oplus A_6 & \langle 1/14 \rangle & 8,8,8,8,8,8,128 \\
48 & A_1 \oplus D_6 & A_1^* \oplus \Z/2\Z & 1,2,2,2,2,2,4 \\
49 & A_1 \oplus E_6 & \langle 1/6 \rangle & 2,2,2,2,2,2,8 \\
52 & A_1^2 \oplus D_5 & \langle 1/4 \rangle \oplus \Z/2\Z & 2,2,2,2,2,4,8 \\
53 & A_1^2 \oplus A_5 & \langle 1/6 \rangle \oplus \Z/2\Z & 2,2,4,4,4,4,16 \\
57 & A_1^3 \oplus D_4 & A_1^* \oplus (\Z/2\Z)^2 & -1,2,2,2,2,-2,-4 \\
58 & A_1 \oplus A_3^2 & A_1^* \oplus \Z/4\Z & I,I,I,I,2,2,2  \\
60 & A_1^4 \oplus A_3 & \langle 1/4 \rangle \oplus (\Z/2\Z)^2 &  2,2,2,2,-1,-1,4 \\
65 & A_1 \oplus E_7 & \Z/2\Z & 1,1,1,1,1,1,1 \\
70 & A_1 \oplus A_7 & \Z/4\Z & I,I,I,I,I,I,I \\
71 & A_1^2 \oplus D_6 & (\Z/2\Z)^2 &  1,1,1,1,1,1,-1 \\
74 & A_1^2 \oplus A_3^2 & (\Z/2\Z) \oplus (\Z/4\Z) & I,I,I,I,-1,-1,-1 
\end{array} 
$$
Here $I = \sqrt{-1}$.

\begin{remark}
  For the examples in lines $58$, $70$ and $74$ of the table, one can
  show that it is not possible to define a rational elliptic surface
  over $\Q$ in the form we have assumed, such that all the
  specializations $s_i, r$ are rational. However, there do exist
  examples with all sections defined over $\Q$, not in the chosen
  Weierstrass form.

The surface with Weierstrass equation
$$
y^2 + xy + \frac{(c^2-1)(t^2-1)}{16} y = x^3 +  \frac{(c^2-1)(t^2-1)}{16} x^2
$$ 
has a $4$-torsion section $(0,0)$ and a non-torsion section $\big(
(c+1)(t^2-1)/8, (c+1)^2(t-1)^2(t+1)/32 \big)$ of height $1/2$, as well
as two reducible fibers of type $\I_4$ and a fiber of type $\I_2$. It
is an example of type $58$.

The surface with Weierstrass equation
$$
y^2 + xy + t^2 y = x^3 + t^2 x^2
$$ 
has a $4$-torsion section $(0,0)$, and reducible fibers of types
$\I_8$ and $\I_2$. It is an example of type $70$.

The surface with Weierstrass equation 
$$ 
y^2 + xy - \left(t^2-\frac{1}{16} \right)y = x^3 - \left( t^2 -
\frac{1}{16} \right) x^2
$$ 
has two reducible fibers of type $\I_4$ and two reducible fibers of
type $\I_2$. It also has a $4$-torsion section $(0,0)$ and a
$2$-torsion section $\big((4t-1)/8, (4t-1)^2/32\big)$, which generate
the Mordell-Weil group. It is an example of type $74$. This last
example is the universal elliptic curve with $\Z/4\Z \oplus \Z/2\Z$
torsion (compare \cite{Ku}).
\end{remark}

\subsection{Proofs} \label{E7proofs}

We start by considering the coefficients $\eps_i$ of $\Phi_\la(X)$; we
know that $(-1)^i \eps_i$ is simply the $i$'th elementary symmetric
polynomial in the $56$ specializations $s(P_i)$. One shows, either by
explicit calculation with Laurent polynomials, or by calculating the
decomposition of $\bigwedge^i V$ (where $V = V_7$ is the
$56$-dimensional representation of $E_7$), and expressing its
character as polynomials in the fundamental characters, the following
formulae. Some more details are in Section \ref{identities} and the
auxiliary files.
\begin{flalign*}
\eps_1 &= -\chi_7 \\
\eps_2 &= \chi_6 + 1 \\
\eps_3 &= -(\chi_7+\chi_5) \\
\eps_4 &= \chi_6 + \chi_4 + 1 \\
\eps_5 &= -(\chi_6+\chi_3-\chi_1^2+\chi_1+1) \chi_7+(\chi_1-1) \chi_5-\chi_2 \chi_3 \\
\eps_6 &= \chi_1 \chi_7^2+(\chi_5-(\chi_1+1) \chi_2) \chi_7+\chi_6^2 +2 (\chi_3-\chi_1^2+\chi_1+1) \chi_6 \\ 
  & \quad -\chi_2 \chi_5 -(2 \chi_1+1) \chi_4 +\chi_3^2+2 (2 \chi_1+1) \chi_3 \\
  & \quad +\chi_1 \chi_2^2-2 \chi_1^3+\chi_1^2+2 \chi_1+1 \\
\eps_7 &= (-(\chi_1+1) \chi_6+2 \chi_4-2(\chi_1+1) \chi_3+\chi_1^3-3 \chi_1-1) \chi_7 \\
  & \quad -2 (\chi_5-\chi_1 \chi_2) \chi_6 -(\chi_3-\chi_1^2+\chi_1+2) \chi_5 +3 \chi_2 \chi_4 \\
  & \quad -(\chi_1+3) \chi_2 \chi_3-\chi_2^3 +(2 \chi_1-1) \chi_1 \chi_2 .
\end{flalign*}
On the other hand, we can explicitly calculate the first few
coefficients $\eps_i$ of $\Phi_\la(X)$ in terms of the Weierstrass
coefficients, obtaining the following equations. Details for the
method are in Section \ref{resultants}.
\begin{flalign*}
\eps_1 &= -q_3 \\
\eps_2 &= p_2^2 + 12 p_2  - q_2 + 28 \\
\eps_3 &= -3 (2 p_2+9) q_3-q_1+2 p_1 (p_2+5) \\
\eps_4 &= 9 q_3^2-6 p_1 q_3-2 q_2-q_0+8 p_2^3+86 p_2^2+2 (p_0+156) p_2+p_1^2+10 p_0+378 \\ 
\eps_5 &= (8 q_2-20 p_2^2-174 p_2-7 p_0-351) q_3-2 p_1 q_2+6 (p_2+4) q_1 \\
       & \quad +14 p_1 p_2^2 +108 p_1 p_2+2 (p_0+101) p_1 \\
\eps_6 &= 12 (4 p_2+15) q_3^2 -(5 q_1+38 p_1 p_2+140 p_1) q_3 + 4 q_2^2 \\
      & \quad +(16 p_2^2+96 p_2-4 p_0+155) q_2  +2 p_1 q_1+3 (4 p_2+17) q_0 +28 p_2^4+360 p_2^3 \\
      & \quad +(4 p_0+1765) p_2^2 +2 (4 p_1^2+21 p_0+1950) p_2+29 p_1^2+p_0^2+88 p_0+3276 \\
\eps_7 &= -36 q_3^3+42 p_1 q_3^2 +(4 q_2-20 q_0-56 p_2^3-628 p_2^2-14 (p_0+168) p_2-16 p_1^2 \\
      & \quad -46 p_0-2925) q_3+(3 q_1+6 p_1 p_2+20 p_1) q_2+(21 p_2^2+162 p_2-p_0+323) q_1 \\
     & \quad +6 p_1 q_0+42 p_1 p_2^3 +448 p_1 p_2^2+2 (p_0+799) p_1 p_2+2 p_1^3+6 (p_0+316) p_1 .
\end{flalign*}

Equating the two expressions we have obtained for each $\eps_i$, we
get a system of seven equations, the first being
$$
-\chi_7 = -q_3.
$$
We label these equations $(1), \dots, (7)$. The last few of these
polynomial equations are somewhat complicated, and so to obtain a few
simpler ones, we may consider the $126$ sections of height $2$, which
we analyze as follows. Substituting
\begin{align*}
x &= at^2 + bt + c \\
y &= dt^3 + et^2 + ft + g
\end{align*}
in to the Weierstrass equation, we get another system of equations:
\begin{align*}
a^3 &= d^2 + ad \\
3a^2b &= (2d + a) e + bd \\
a(p_2 + 3 ac + 3b^2) &= (2d + a) f + e^2 + be + cd + 1 \\
q_3 + b p_2 +a p_1 + 6abc+b^3 &= (2d + a) g + (2e + b) f + ce \\
q_2 + c p_2 + b p_1 + a p_0 + 3 ac^2 +  3b^2 c &= (2e + b) g + f^2 + c f \\
q_1 + c p_1 + b p_0 + 3 b c^2 &= (2f + c) g \\
q_0 + c p_0 + c^3 &= g^2.
\end{align*}
The specialization of such a section at $t = \infty$ is $1 +
a/d$. Setting $d = ar$, we may as before eliminate other variables to
obtain an equation of degree $126$ for $r$. Substituting $r =
1/(u-1)$, we get a monic polynomial $\Psi_\la(X) = 0$ of degree $126$
for $u$. Note that the roots are given by elements of the form
$$
\frac{s_i}{r} , \frac{r}{s_i} \textrm{ for } 1 \leq i \leq 7, \, \frac{s_i}{s_j} \textrm{ for } 1 \leq i \neq j \leq 7 \, \textrm{ and } \, \frac{s_i s_j s_k}{r}, \frac{r}{s_i s_j s_k} \textrm{ for } 1 \leq i <  j < k \leq 7.
$$ 

As before, we can write the first few coefficients $\eta_i$ of
$\Psi_\la$ in terms of $\la = (p_0, \dots, q_3)$, as well as in terms
of the characters $\chi_j$, obtaining some more relations. We will
only need the first two:
\begin{align*}
-\chi_1 + 7 &= \eta_1 = -18 - 6 p_2  \\
 -6 \chi_1 + \chi_3 + 28 &= \eta_2 = p_0 + 72 p_2 + 13 p_2^2 - q_2 + 99 
\end{align*}
which we call $(1')$ and $(2')$ respectively.

Now we consider the system of six equations $(1)$ through $(4)$,
$(1')$ and $(2')$. These may be solved for $(p_2, p_0, q_3, q_2, q_1,
q_0)$ in terms of the $\chi_j$ and $p_1$. Substituting this solution
into the other three relations $(5)$, $(6)$ and $(7)$, we obtain three
equations for $p_1$, of degrees $1$, $2$ and $3$ respectively. These
have a single common factor, linear in $p_1$, which we then
solve. This gives us the proof of Theorem \ref{explicitE7}.

The proof of Theorem \ref{excellentE7} is now straightforward. Part
(1) asserts that the image of $\rho_\la$ is surjective on to $W(E_7)$:
this follows from a standard Galois theoretic argument as follows. Let
$F$ be the fixed field of $W(E_7)$ acting on $k_\la = \Q(\la)(s_1,
\dots, s_6, r) = \Q(s_1, \dots, s_6,r)$, where the last equality follows
from the explicit expression of $\la = (p_0, \dots, q_3)$ in terms of
the $\chi_i$, which are in $\Q(s_1, \dots, s_6, r)$.  Then we have
that $k_0 \subset F$ since $p_0, \dots, q_3$ are polynomials in the
$\chi_i$ with rational coefficients, and the $\chi_i$ are manifestly
invariant under the Weyl group. Therefore $[k_\la: k_0] \geq [k_\la:
  F] = |W(E_7)|$, where the latter equality is from Galois
theory. Finally, $[k_\la: k_0] \leq |\Gal(k_\la/k_0)| \leq |W(E_7)|$,
since $\Gal(k_\la/k_0) \hookrightarrow W(E_7)$. Therefore, equality is
forced.

Another way to see that the Galois group is the full Weyl group is to
demonstrate it for a specialization, such as Example
\ref{bigexampleE7}, and use \cite[Section 9.2, Prop 2]{Se}.

Next, let $Y$ be the toric hypersurface given by $s_1 \dots s_7 =
r^3$. Its function field is the splitting field of $\Phi_\la(X)$, as
we remarked above. We have seen that $\Q(Y)^{W(E_7)} = k_0 =
\Q(\la)$. Since $\Phi_\la(X)$ is a monic polynomial with coefficients
in $\Q[\la]$, we have that $\Q[Y]$ is integral over
$\Q[\la]$. Therefore $\Q[Y]^{W(E_7)}$ is also integral over $\Q[\la]$,
and contained in $\Q(Y)^{W(E_7)} = k_0 = \Q(\la)$. Since $\Q[\la]$ is
a polynomial ring, it is integrally closed in its ring of
fractions. Therefore $\Q[Y]^{W(E_7)} \subset \Q[\la]$.

We also know $\Q[\chi] = \Q[\chi_1, \dots, \chi_7] \subset
\Q[Y]^{W(E_7)}$, since the $\chi_j$ are invariant under the
Weyl group. Therefore, we have
$$
\Q[\chi] \subset \Q[Y]^{W(E_7)} \subset \Q[\la]
$$ and Theorem \ref{explicitE7}, which says $\Q[\chi] = \Q[\la]$,
implies that all these three rings are equal. This completes the proof
of Theorem \ref{excellentE7}. 

\begin{remark}
Note that this argument gives an independent proof of the fact that
the ring of multiplicative invariants for $W(E_7)$ is a polynomial
ring over $\chi_1, \dots, \chi_7$. See \cite[Th\'eor\`eme VI.3.1 and
  Exemple 1]{B} or \cite[Theorem 3.6.1]{L} for the classical proof
that the Weyl-orbit sums of a set of fundamental weights are a set of
algebraically independent generators of the multiplicative invariant
ring; from there to the fundamental characters is an easy exercise.
\end{remark}

\begin{remark}
  Now that we have found the explicit relation between the Weierstrass
  coefficients and the fundamental characters, we may go back and
  explore the ``genericity condition'' for this family to have
  Mordell-Weil lattice $E_7^*$. To do this we compute the discriminant
  of the cubic in $x$, after completing the square in $y$, and take
  the discriminant with respect to $t$ of the resulting polynomial of
  degree $10$. A computation shows that this discriminant factors as
  the cube of a polynomial $A(\la)$ (which vanishes exactly when the
  family has a fiber of additive reduction, generically type $\II$),
  times a polynomial $B(\la)$, whose zero locus corresponds to the
  occurrence of a reducible multiplicative fiber. In fact, we
  calculate (for instance, by evaluating the split case), that
  $B(\la)$ is the product of $(e^{\alpha} - 1)$, where $\alpha$ runs
  over the $126$ roots of $E_7$. We deduce by further analyzing the
  type $\II$ case that the condition to have Mordell-Weil lattice
  $E_7^*$ is that
$$ 
\prod (e^{\alpha} - 1) = \Psi_\la(1) \neq 0.
$$ Note that this is essentially the expression which occurs in Weyl's
denominator formula. In addition, the condition for having only
multiplicative fibers is that $\Psi_\la(1)$ and $A(\la)$ both be
non-zero.
\end{remark}

Finally, the proof of Theorem \ref{E7degen} follows immediately from
the discussion in \cite{Sh6, Sh7} (compare \cite[Section 2.3]{Sh7} for
the additive reduction case).

\section{The $E_8$ case}

\subsection{Results} 
Finally, we show a multiplicative excellent family for the Weyl group
of $E_8$. It is given by the Weierstrass equation
$$
y^2 = x^3 + t^2 \, x^2 + (p_0 + p_1 t + p_2 t^2) \,x + (q_0 + q_1 t + q_2 t^2 + q_3 t^3 + q_4 t^4 + t^5).
$$ For generic $\lambda = (p_0, \dots, p_2, q_0, \dots, q_4)$, this
rational elliptic surface $X_\la$ has no reducible fibers, only nodal
$\I_1$ fibers at twelve points, including $t = \infty$. We will use
the specialization map at $\infty$. The Mordell-Weil lattice $M_\la$
is isomorphic to the lattice $E_8$. Any rational elliptic surface with
a multiplicative fiber of type $\I_1$ may be put in the above form
(over a small degree algebraic extension of the base field), after a
fractional linear transformation of $t$ and Weierstrass
transformations of $x,y$.

\begin{lemma}
The smooth part of the special fiber is isomorphic to the group scheme
$\G_m$. The identity component is the non-singular part of the curve
$y^2 = x^3 + x^2$. A section of height $2$ has $x$- and
$y$-coordinates polynomials of degrees $2$ and $3$ respectively, and
its specialization at $t = \infty$ may be taken as $\lim_{t \to
  \infty} (y + tx)/(y-tx) \in k^*$.
\end{lemma}

The proof of the lemma is similar to that in the $E_7$ case (and
simpler!), and we omit it.

There are $240$ sections of minimal height $2$, with $x$ and $y$
coordinates of the form
\begin{align*}
x &= gt^2 + at + b \\
y &= ht^3 + ct^2 + dt + e.
\end{align*}
Under the identification with $\G_m$ of the special fiber of the
N\'eron model, this section goes to $(h+g)/(h-g)$. Substituting the
above formulas for $x$ and $y$ into the Weierstrass equation, we get
the following system of equations.
\begin{align*}
h^2 &= g^3 + g^2 \\
2ch &= 3ag^2+2ag+1 \\
2dh+c^2 &= q_4+g p_2+3 b g^2+(2 b+3 a^2) g+a^2 \\
2eh+2cd &= q_3+ap_2+gp_1+6abg+2ab+a^3 \\
2ce+d^2  &= q_2+bp_2+ap_1+gp_0+3b^2g+b^2+3a^2b \\
2de &= q_1+b p_1+a p_0+3 a b^2 \\
e^2 &= q_0+b p_0+b^3.
\end{align*}

Setting $h = gu$, we eliminate other variables to obtain an equation
of degree $240$ for $u$. Finally, substituting in $u = (v + 1)/(v-1)$,
we get a monic reciprocal equation $\Phi_\la(X) = 0$ satisfied by $v$,
with coefficients in $\Z[\lambda] = \Z[p_0, \dots, p_2, q_0, \dots,
  q_4]$. We have
$$
\Phi_\la(X) = \prod_{i=1}^{240} (X - s(P)) = X^{240} + \eps_1 X^{239} + \dots + \eps_1 X + \eps_0, 
$$ where $P$ ranges over the $240$ minimal sections of height $2$. It
is clear that $a,\dots,h$ are rational functions of $v$, with
coefficients in $k_0$.

We have a Galois representation on the Mordell-Weil lattice 
$$
\rho_\la: \Gal(k/k_0) \rightarrow \Aut(M_\la) \cong \Aut(E_8).
$$ 

Here $\Aut(E_8) \cong W(E_8)$, the Weyl group of type $E_8$. The {\em
  splitting field} of $M_\la$ is the fixed field $k_\la$ of
$\mathrm{Ker}(\rho_\la)$. By definition, $\Gal(k_\la/k_0) \cong
\textrm{Im}(\rho_\la)$. The splitting field $k_\la$ is equal to the
splitting field of the polynomial $\Phi_\la(X)$ over $k_0$, since the
Mordell-Weil group is generated by the $240$ sections of smallest
height $P_i = (g_it^2 + a_i t + b_i, h_i t^3 + c_i t^2 + d_i t +
e_i)$. We also have that
$$
k_\la = k_0(P_1, \dots, P_{240}) = k_0(v_1, \dots, v_{240}).
$$

\begin{theorem} \label{excellentE8}
  Assume that $\lambda$ is generic over $\Q$, i.e the coordinates
  $p_0, \dots, q_4$ are algebraically independent over $\Q$. Then
\begin{enumerate}
\item $\rho_\la$ induces an isomorphism $\Gal(k_\la / k_0) \cong W(E_8)$.
\item The splitting field $k_\la$ is a purely transcendental extension
  of $\Q$, and is isomorphic to the function field $\Q(Y)$ of the
  toric hypersurface $Y \subset \G_m^9$ defined by $s_1 \cdots s_8 =
  r^3$. There is an action of $W(E_8)$ on $Y$ such that
  $\Q(Y)^{W(E_8)} = k_\la^{W(E_8)} = k_0$.
\item The ring of $W(E_8)$-invariants in the affine coordinate ring
  $$
  \Q[Y] = \Q[s_i, r, 1/s_i, 1/r]/(s_1 \dots s_8 - r^3) \cong \Q[s_1, \dots, s_7, r, s_1^{-1}, \dots, s_7^{-1}, r^{-1}]
  $$ 
is equal to the polynomial ring $\Q[\la]$:
$$
\Q[Y]^{W(E_8)} = \Q[\la] = \Q[p_0, p_1, p_2, q_0, q_1, q_2, q_3, q_4].
$$
\end{enumerate}
\end{theorem}

As in the $E_7$ case, we prove an explicit, invertible polynomial
relation between the Weierstrass coefficients $\lambda$ and the
fundamental characters for $E_8$. Let $V_1, \dots, V_8$ be the
fundamental representations of $E_8$, and $\chi_1, \dots, \chi_8$
their characters as labelled below.

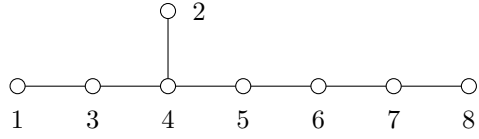
\begin{center}
\begin{tikzpicture}

\draw (0,0)--(6,0);
\draw (2,0)--(2,1);

\fill [white] (0,0) circle (0.1);
\fill [white] (1,0) circle (0.1);
\fill [white] (2,0) circle (0.1);
\fill [white] (3,0) circle (0.1);
\fill [white] (4,0) circle (0.1);
\fill [white] (5,0) circle (0.1);
\fill [white] (6,0) circle (0.1);
\fill [white] (2,1) circle (0.1);

\draw (0,0) circle (0.1);
\draw (1,0) circle (0.1);
\draw (2,0) circle (0.1);
\draw (3,0) circle (0.1);
\draw (4,0) circle (0.1);
\draw (5,0) circle (0.1);
\draw (6,0) circle (0.1);
\draw (2,1) circle (0.1);

\draw (0,-0.2) [below] node{$1$};
\draw (1,-0.2) [below] node{$3$};
\draw (2,-0.2) [below] node{$4$};
\draw (3,-0.2) [below] node{$5$};
\draw (4,-0.2) [below] node{$6$};
\draw (5,-0.2) [below] node{$7$};
\draw (6,-0.2) [below] node{$8$};
\draw (2.2,1) [right] node{$2$};

\end{tikzpicture}
\end{center}

Again, for the set of generators of $E_8$, we choose (as in \cite{Sh5})
vectors $v_1, \dots, v_8, u$ with $\sum v_i = 3 u$ and let $s_i$
correspond to $v_i$ and $r$ to $u$, so that $\prod s_i = r^3$.  The
$240$ roots of $\Phi_\la(X)$ are given by
\begin{align*}
s_i, \frac{1}{s_i} \textrm{ for } 1 \leq i \leq 8, &  \quad \frac{s_i}{s_j} \textrm{ for } 1 \leq i \neq j \leq 8, \\
\frac{s_i s_j}{r}, \frac{r}{s_i s_j} \textrm{ for } 1 \leq i < j \leq 8 & \quad \textrm{ and } \quad \frac{s_i s_j s_k}{r}, \frac{r}{s_i s_j s_k} \textrm{ for } 1 \leq i < j < k \leq 8.
\end{align*}

The characters $\chi_1, \dots, \chi_7$ lie in the ring of Laurent
polynomials $\Q[s_i, r, 1/s_i, 1/r]$, and are invariant under the
multiplicative action of the Weyl group on this ring of Laurent
polynomials. The $\chi_i$ may be explicitly computed using the
software \texttt{LiE}, as indicated in Section \ref{identities} and
the auxiliary files.

\begin{theorem} \label{explicitE8}
For generic $\lambda$ over $\Q$, we have 
$$
\Q[\chi_1, \dots, \chi_8] = \Q[p_0, p_1, p_2, q_0, q_1, q_2, q_3, q_4].
$$
The transformation between these sets of generators is
\begin{flalign*}
\chi_1 &= -1600 q_4+1536 p_2+3875 \\
\chi_2 &= 2 (-45600 q_4+12288 q_3+40704 p_2-16384 p_1+73625) \\
\chi_3 &= 64 (14144 q_4^2-72 (384 p_2+1225) q_4+11200 q_3-4096 q_2+13312 p_2^2 \\
       & \quad +87072 p_2-17920 p_1 +16384 p_0+104625 ) \\
\chi_4 &= -91750400 q_4^3+12288 (25600 p_2+222101) q_4^2 -256 (4530176 q_3-65536 q_2 \\
       & \quad +1392640 p_2^2+21778944 p_2-8159232 p_1+2621440 p_0  +34773585) q_4 \\
   & \quad  + 32 ( 4718592 q_3^2+ 384 (80896 p_2-32768 p_1+225379) q_3-29589504 q_2 \\
    & \quad  +30408704 q_1 -33554432 q_0+4194304 p_2^3+88129536 p_2^2 \\ 
    & \quad  -64 (876544 p_1-262144 p_0-4399923) p_2+8388608 p_1^2-133996544 p_1 \\
    & \quad  +65175552 p_0+215596227  ) 
\end{flalign*}
\begin{flalign*}
\chi_5 &= 24760320 q_4^2-64 (106496 q_3+738816 p_2-163840 p_1+2360085) q_4 \\
       & \quad  +12288 (512 p_2+4797) q_3-30670848 q_2+16777216 q_1+20250624 p_2^2 \\
       &\quad  -512 (16384 p_1-235911) p_2-45154304 p_1+13631488 p_0+146325270 \\
\chi_6 &= 110592 q_4^2-1536 (128 p_2+1235) q_4+724992 q_3-262144 q_2+65536 p_2^2 \\
       & \quad  +1062912 p_2-229376 p_1+2450240 \\
\chi_7 &= -4 (3920 q_4-1024 q_3-1152 p_2-7595) \\
\chi_8 &= -8 (8 q_4-31).
\end{flalign*}
\end{theorem}
\begin{remark}
  We omit the inverse for conciseness here; it is easily computed in a
  computer algebra system and is available in the auxiliary files.
\end{remark}
\begin{remark}
  As before, our explicit formulas are compatible with those in
  \cite{ES}. Also, the proof of Theorem \ref{excellentE8} gives
  another proof of the fact that the multiplicative invariants for
  $W(E_8)$ are freely generated by the fundamental characters (or by
  the orbit sums of the fundamental weights).
\end{remark}
\begin{example} \label{splitexampleE8}
Let $\mu = (2 \cdot 3 \cdot 5 \cdot 7 \cdot 11 \cdot 13 \cdot 17 \cdot 19) = 9699690.$
\begin{align*}
q_4 &= -2243374456559366834339/(2^5 \cdot \mu^2) \\
q_3 &= 430800343129403388346226518246078567/(2^{11} \cdot \mu^3) \\
q_2 &= 72555101947649011127391733034984158462573146409905769/(2^{22}\cdot 3^2 \cdot \mu^4) \\
q_1 &= (-12881099305517291338207432378468368491584063772556981164919245 \\
    & \qquad 30489)/(2^{29} \cdot 3 \cdot  \mu^5) \\
q_0 &= (8827176793323619929427303381485459401911918837196838709750423283 \\
    & \qquad 443360357992650203)/(2^{42} \cdot 3^3 \cdot \mu^6) \\
p_2 &= 146156773903879871001810589/(2^9 \cdot 3 \cdot \mu^2) \\
p_1 &= -24909805041567866985469379779685360019313/(2^{20} \cdot \mu^3) \\
p_0 &= 14921071761102637668643191215755039801471771138867387/(2^{23} \cdot 3 \cdot  \mu^4) \\
\end{align*}
These values give an elliptic surface for which we have $r = 2, s_1 =
3, s_2 = 5, s_3 = 7, s_4 = 11, s_5 = 13, s_6 = 17, s_7 = 19$, the
simplest choice of multiplicatively independent elements. Here, the
specializations of a basis are given by $v \in
\{3,5,7,11,13,17,19,15/2\}$. Once again, we list the $x$-coordinates
of the corresponding sections, and leave the remainder of the
verification to the auxiliary files.

\begin{align*}
x(P_1) &= 3 t^2 - (99950606190359/620780160)t  \\
 & \quad + 4325327557647488120209649813/2642523476911718400 \\ 
x(P_2) &= (5/4)t^2 - (153332163637781/1655413760)t \\
 & \quad + 5414114237697608646836821/5138596941004800 \\ 
x(P_3) &= (7/9)t^2 - (203120672689603/2793510720)t \\
 & \quad + 6943164348569130636788638639/7927570430735155200 \\
x(P_4) &= (11/25)t^2 - (8564057914757/147804800)t \\
 & \quad + 115126372233675800396600989/155442557465395200 
\end{align*}
\begin{align*}
x(P_5) &= (13/36)t^2 - (347479008951469/6385167360)t \\
 & \quad + 157133607680949617374030405417/221971972060584345600 \\
x(P_6) &= (17/64)t^2 - (1327421017414859/26486620160)t \\
 & \quad + 5942419292933021418457517303/8901131711702630400 \\ 
x(P_7) &= (19/81)t^2 - (489830985359431/10056638592)t \\
 & \quad + 46685137201743696441477454951/71348133876616396800 \\ 
x(P_8) &= (120/169)t^2 - (30706596009257/440806080)t \\
 & \quad + 76164443074828743662165466409/55823308449760051200.
\end{align*}

\end{example}

\begin{example} \label{bigexampleE8} The value $\la = \la_0 :=
  (1,1,1,1,1,1,1,1)$ gives rise to an explicit polynomial $g(X) =
  \Phi_{\la_0}(X)$, for which we can show that the Galois group is
  $W(E_8)$, as follows. The reduction of $g(X)$ modulo $79$ shows that
  $\Frob_{79}$ has cycle decomposition of type $(4)^2(8)^{29}$, and
  similarly, $\Frob_{179}$ has cycle decomposition of type
  $(15)^{16}$. We deduce, as in \cite[Section 3]{JKZ} or \cite{Sh8},
  that the Galois group is the entire Weyl group. Since the
  coefficients of $g(X)$ are large, we do not display it here, but it
  is included in the auxiliary files.
\end{example}

As in the case of $E_7$, we can also describe degenerations of this
family of rational elliptic surfaces $X_\la$ by the method of
``vanishing roots'', where we drop the genericity assumption, and
consider the situation where the elliptic fibration might have
additional reducible fibers. Let $\psi: Y \rightarrow \A^8$ be the
finite surjective morphism associated to
$$
\Q[p_0, \dots, q_4] \hookrightarrow \Q[Y] \cong \Q[s_1, \dots, s_7, r, s_1^{-1}, \dots, s_7^{-1}, r^{-1}].
$$ 
For $\xi = (s_1, \dots, s_8, r) \in Y$, let the multiset $\Pi_\xi$
consist of the $240$ elements $s_i$ and $1/s_i$ (for $1 \leq i \leq
8$), $s_i/s_j$ (for $1 \leq i \neq j \leq 8$), $s_i s_j/r$ and $r/(s_i
s_j)$ (for $1 \leq i < j \leq 8$) and $s_i s_j s_k/r$ and $r/(s_i s_j
s_k)$ for $1 \leq i < j < k \leq 8$, corresponding to the $240$ roots
of $E_8$. Let $2\nu(\xi)$ be the number of times $1$ appears in
$\Pi_\xi$, which is also the multiplicity of $1$ as a root of
$\Phi_\la(X)$, with $\la = \psi(\xi)$. We call the associated roots of
$E_8$ the {\em vanishing roots}, in analogy with vanishing cycles in
the deformation of singularities. By abuse of notation we label the
 rational elliptic surface $X_\la$ as $X_\xi$.

\begin{theorem} \label{E8degen}
  The surface $X_\xi$ has new reducible fibers (necessarily at $t \neq
  \infty$) if and only if $\nu(\xi) > 0$. The number of roots in the
  root lattice $T_{\textrm{new}}$ is equal to $2\nu(\xi)$, where
  $T_{\textrm{new}} := \oplus_{v \neq \infty} T_v$ is the new part of
  the trivial lattice.
\end{theorem}

We may use this result to produce specializations with trivial lattice
corresponding to most of the entries of \cite{OS}, and a nodal
fiber. Below, we list those types which are not already covered by
\cite{Sh2, Sh4} or our examples for the $E_7$ case, which have an
$\I_2$ fiber.

$$
\begin{array}{cccl}
\textrm{Type in \cite{OS}} & \textrm{Fibral lattice} & \textrm{MW group} & \{s_1, \dots, s_6, r \} \\
1 & 0 & E_8 & 3,5,7,11,13,17,19,2 \\
5 & A_3 & D_5^* & 2,2,2,2,5,7,11,3 \\
8 & A_4 & A_4^* & 2,2,2,2,2,5,7,3 \\
15 & A_5 & A_2^* \oplus A_1^* &  2,2,2,2,2,2,5,3 \\
16 & D_5 & A_3^* &  2,3,3,3,3,3,5,18 \\
25 & A_6 & \frac{1}{7} \footnotesize{ \left( \begin{array}{cc} 4 & -1 \\ -1 & 2 \end{array} \right) } & 2,2,2,2,2,2,2,3 \\
26 & D_6 & A_1^{*2} &  2,3,3,3,3,3,3,18 \\
35 & A_3^2 & A_1^{*2} \oplus \Z/2\Z &  2,-1/2,3,3,3,1,1,-3 
\end{array} 
$$ 

$$
\begin{array}{cccl}
\textrm{Type in \cite{OS}} & \textrm{Fibral lattice} & \textrm{MW group} & \{s_1, \dots, s_6, r \} \\
36 & A_3^2 & \langle 1/4 \rangle &  8,8,8,8,27,27,27,1296 \\
43 & E_7 & A_1^* & 2,2,2,2,2,2,2,8 \\
44 & A_7 & A_1^* \oplus \Z/2\Z & 2,2,2,2,2,2,2,-8 \\
45 & A_7 & \langle 1/8 \rangle &  8,8,8,8,8,8,8,256 \\
46 & D_7 & \langle 1/4 \rangle &  2,4,4,4,4,4,4,32 \\
54 & A_3 \oplus D_4 & \langle 1/4 \rangle \oplus \Z/2\Z & 2,-1,-1,-1,-1,1,1,2 \\
55 & A_3 \oplus A_4 & \langle 1/20 \rangle & 16,16,16,16,32,32,32,4096 \\
62 & E_8 & 0 &  1,1,1,1,1,1,1,1 \\
63 & A_8 & \Z/3\Z &  1,1,1,1,1,1,1, \zeta_3 \\
64 & D_8 & \Z/2\Z &  1,1,1,1,1,1,1,-1 \\
67 & A_4^2 & \Z/5\Z & 1,1,1,1,\zeta_5,\zeta_5,\zeta_5,\zeta_5^3 \\
72 & A_3 \oplus D_5 & \Z/4\Z & 1,1,1,I,I,I,I,-I 
\end{array} 
$$ 
Here $\zeta_3$, $I$ and $\zeta_5$ are primitive third, fourth and
fifth roots of unity.

\begin{remark}
  As before, for the examples in lines $63$, $67$ and $72$ of the
  table, one can show that it is not possible to define a rational
  elliptic surface over $\Q$ in the form we have assumed, such that
  all the specializations $s_i, r$ are rational. However, there do
  exist examples with all sections defined over $\Q$, not in the
  chosen Weierstrass form.

The surface with Weierstrass equation 
$$
y^2 + xy + t^3 y = x^3
$$ 
has a $3$-torsion point $(0, 0)$ and a fiber of type $\I_9$. It is
an example of type $63$.

The surface with Weierstrass equation
$$
y^2 + (t+1)xy + ty = x^3 + tx^2
$$ 
has a $5$-torsion section $(0,0)$ and two fibers of type $\I_5$. It
is an example of type $67$.

The surface with Weierstrass equation 
$$
y^2 + txy + \frac{t^2(t - 1)}{16} y  = x^3 + \frac{t(t - 1)}{16}x^2
$$
has a $4$-torsion section $(0,0)$, and two
fibers of types $\I_4$ and $\I_1^*$. It is an example of type $72$.
\end{remark}

\begin{remark}
Our tables and the one in \cite{Sh4} cover all the cases of \cite{OS},
except lines $9$, $27$ and $73$ of the table, with trivial lattice
$D_4$, $E_6$ and $D_4^2$ respectively. Since these have fibers with
additive reduction, examples for them may be directly constructed
using the families in \cite{Sh2}. For instance, the elliptic surface
$$
y^2 = x^3 - x t^2
$$ 
has two fibers of type $\I_0^*$ and Mordell-Weil group
$(\Z/2\Z)^2$. This covers line $73$ of the table. For the other two
cases, we refer the reader to the original examples of additive
reduction in Section 3 of \cite{Sh2}.
\end{remark}

\subsection{Proofs}

The proof proceeds analogously to the $E_7$ case: with two
differences: we only have one polynomial $\Phi_\la(X)$ to work with
(as opposed to having $\Phi_\la(X)$ and $\Psi_\la(X)$), and the
equations are a lot more complicated.

We first write down the relation between the coefficients $\eps_i$, $1
\leq i \leq 9$, and the fundamental invariants $\chi_j$; as before, we
postpone the proofs to the auxiliary files and outline the idea in
Section \ref{identities}. Second, we write down the coefficients
$\eps_i$ in terms of $\lambda = (p_0, \dots, p_2, q_0, \dots, q_4)$;
see Section \ref{resultants} for an explanation of how this is carried
out. In the interest of brevity, we do not write out either of these
sets of equations, but relegate them to the auxiliary computer
files. Finally, setting the corresponding expressions equal to each
other, we obtain a set of equations $(1)$ through $(9)$.

To solve these equations, proceed as follows: first use $(1)$ through
$(5)$ to solve for $q_0, \dots, q_4$ in terms of $\chi_j$ and $p_0,
p_1, p_2$. Substituting these in to the remaining equations, we obtain
$(6')$ through $(9')$. These have low degree in $p_0$, which we
eliminate, obtaining equations of relatively small degrees in $p_1$
and $p_2$. Finally, we take resultants with respect to $p_1$,
obtaining two equations for $p_2$, of which the only common root is
the one listed above. Working back, we solve for all the other
variables, obtaining the system above and completing the proof of
Theorem \ref{explicitE8}. The deduction of Theorem \ref{excellentE8}
now proceeds exactly as in the case of $E_7$.

\begin{remark}
  As in the $E_7$ case, once we find the explicit relation between the
  Weierstrass coefficients and the fundamental characters, we may go
  back and explore the ``genericity condition'' for this family to
  have Mordell-Weil lattice isomorphic to $E_8$. To do this we compute
  the discriminant of the cubic in $x$, after completing the square in
  $y$, and take the discriminant with respect to $t$ of the resulting
  polynomial of degree $11$. A computation shows that this
  discriminant factors as the cube of a polynomial $A(\la)$ (which
  vanishes exactly when the family has a fiber of additive reduction,
  generically type $\II$), and the product of $(e^{\alpha} - 1)$,
  where $\alpha$ runs over minimal vectors of $E_8$. Again, the
  genericity condition to have Mordell-Weil lattice exactly $E_8$ is
  just the nonvanishing of
$$
\Phi_\la(1) = \prod (e^\alpha - 1),
$$ the expression which occurs in the Weyl denominator
formula. Furthermore, the condition to have only multiplicative fibers
is that $\Phi_\la(1)A(\la) \neq 0$.
\end{remark}

As before, the proof of Theorem \ref{E8degen} follows immediately from
the results of \cite{Sh6, Sh7}, by degeneration from a flat family.

\section{Resultants, Interpolation and Computations} \label{resultants}

We now explain a computational aid, used in obtaining the equations
expressing the coefficients of $\Phi_\la$ (for $E_8$) or $\Psi_\la$
(for $E_7$) in terms of the Weierstrass coefficients of the associated
family of rational elliptic surfaces. We illustrate this using the
system of equations obtained for sections of the $E_8$ family:
\begin{align*}
h^2 &= g^3 + g^2 \\
2ch &= 3ag^2+2ag+1 \\
c^2 + 2dh &= q_4+gp_2+3bg^2+(2b+3a^2)g+a^2 \\
2eh + 2cd &= q_3+ap_2+gp_1+6abg+2ab+a^3 \\
2ce + d^2 &= q_2+bp_2+ap_1+gp_0+3b^2g+b^2+3a^2b \\
2de &= q_1+bp_1+ap_0+3ab^2 \\
e^2 &= q_0 + b p_0 + b^3.
\end{align*}
Setting $h = gu$ and solving the first equation for $g$ we have $g =
u^2 - 1$. We solve the next three equations for $c,d,e$
respectively. This leaves us with three equations $R_1(a,b,u) =
R_2(a,b,u) = R_3(a,b,u) = 0$. These have degrees $2,2,3$ respectively
in $b$. Taking the appropriate linear combination of $R_1$ and $R_2$
gives us an equation $S_1(a,b,u) = 0$ which is linear in
$b$. Similarly, we may use $R_1$ and $R_3$ to obtain another equation
$S_2(a,b,u) = 0$, linear in $b$. We write
\begin{align*}
S_1(a,b,u) &= s_{11}(a,u) b + s_{10} (a,u) \\
S_2(a,b,u) &= s_{21}(a,u) b + s_{20} (a,u) \\
R_1(a,b,u) &= r_2 (a,u) b^2 + r_1 (a,u) b + r_0 (a,u).
\end{align*}
The resultant of the first two polynomials gives us an equation
$$
T_1 (a,u) = s_{11} s_{20} - s_{10} s_{21} = 0
$$
while the resultant of the first and third gives us 
$$
T_2 (a,u) = r_2 s_{10}^2 - r_1 s_{10} s_{11} + r_0 s_{11}^2 = 0.
$$

Finally, we substitute $u =(v+1)/(v-1)$ throughout, obtaining two
equations $\tilde{T}_1(a,v) = 0$ and $\tilde{T}_2(a,v) = 0$.

Next, we would like to compute the resultant of $\tilde{T}_1(a,v)$ and
$\tilde{T}_2(a,v)$, which have degrees $8$ and $9$ with respect to $a$,
to obtain a single equation satisfied by $v$. However, the polynomials
$\tilde{T}_1$ and $\tilde{T}_2$ are already fairly large (they take a
few hundred kilobytes of memory), and since their degree in $a$ is not
too small, it is beyond the current reach of computer algebra systems
such as \texttt{gp/PARI} or \texttt{Magma} to compute their
resultant. It would take too long to compute their resultant, and
another issue is that the resultant would take too much memory to
store, certainly more than is available on the authors' computer
systems (for instance, it would take more than 16GB of memory).

To circumvent this issue, what we shall do is to use several
specializations of $\lambda$ in $\Q^8$. Once we specialize, the
polynomials take much less space to store, and the computations of the
resultants becomes tremendously easier.  Since the resultant can be
computed via the Sylvester determinant
$$
\left|
\begin{array}{ccccccccccc}
a_8  & \dots & a_2 & a_1 & a_0 & 0 & 0 & \dots & 0 \\
0 & a_8 & \dots & a_2 & a_1 & a_0 & 0 & \dots & 0  \\
\vdots & \ddots & \ddots & & & \ddots & \ddots & \ddots & \vdots \\
0 & \dots & 0  & a_8 & \dots & a_2 & a_1 &  a_0 & 0 \\
0 & \dots & 0 & 0 & a_8 &  \dots & a_2 & a_1 & a_0 \\
b_9 & b_8 & \dots & b_2 & b_1 & b_0 & 0 & \dots & 0 \\
0 & b_9 & b_8 & \dots & b_2 & b_1 & b_0 & \ddots & \vdots \\
\vdots & \ddots & \ddots  & & & & \ddots & \ddots & 0 \\
0 & \dots & 0 & b_9 & b_8 & \dots & b_2 & b_1 & b_0 
\end{array}
\right|
$$
where $\tilde{T}_1(a,v) = \sum a_i(v) a^i$ and $\tilde{T}_2(a,v) =
\sum b_i(v) a^i$, we see that the resultant is a polynomial $Z(v) =
\sum z_i v^i$ with coefficients $z_i$ being polynomials in the
coefficients of the $a_i$ and the $b_j$, which happen to be elements
of $\Q[\la]$ (recall that $\lambda = (p_0, \dots, p_2, q_0, \dots,
q_4)$). Furthermore, we can bound the degrees $m_i(j)$ of $z_i(v)$
with respect to the $j$'th coordinate of $\lambda$, by using explicit
bounds on the multidegrees of the $a_i$ and $b_i$. Therefore, by using
Lagrange interpolation (with respect to the eight variables
$\lambda_j$) we can reconstruct $z_i(v)$ from its specializations for
various values of $\lambda$. The same method lets us show that $Z(v)$
is divisible by $v^{22}$ (for instance, by showing that $z_0$ through
$z_{21}$ are zero), and also by $(v+1)^{80}$ (by first shifting $v$ by
$1$ and then computing the Sylvester determinant, and proceeding as
before), as well as by $(v^2 + v + 1)^8$ (this time, using cube roots
of unity). Finally, it is clear that $Z(v)$ is divisible by the square
of the resultant $G(v)$ of $s_{11}$ and $s_{10}$ with respect to
$a$. Removing these extraneous factors, we get a polynomial
$\Phi_\la(v)$ which is monic and reciprocal of degree $240$. We
compute its top few coefficients by this interpolation method.

Finally, we note the interpolation method above is in fact completely
rigorous. Namely, let $\eps_i(\la)$ be the coefficient of $v^i$ in
$\Phi_\la(v)$, with bounds $(m_1, \dots, m_8)$ for its multidegree,
and $\eps'_i(\la)$ the putative polynomial we have computed using
Lagrange interpolation on a set $L_1 \times \dots \times L_8$, where
$L_i = \{ \ell_{i,0}, \dots, \ell_{i,m_i} \}$ for $1 \leq i \leq 8$
are sets of integers chosen generically enough to ensure that $G(v)$
has the correct degree and that $Z(v)$ is not divisible by any higher
powers of $v$, $v+1$ or $v^2 + v + 1$ than in the generic case. Then
since $\eps_j(\ell_{1, i_1}, \dots, \ell_{8, i_8}) = \eps'_j(\ell_{1,
  i_1}, \dots, \ell_{8, i_8})$ for all choices of $i_1, \dots, i_8$,
we see that the difference of these polynomials must vanish.

\section{Representation theory, and some identities in Laurent
  polynomials} \label{identities}

Finally, we demonstrate how to deduce the identities relating the
coefficients of $\Phi_{E_7, \la}(X)$ or $\Psi_{E_7, \la}(X)$ to the
fundamental characters for $E_7$ (and similarly, the coefficients of
$\Phi_{E_8, \la(X)}$ to the fundamental characters of $E_8$).

Conceptually, the simplest way to do this is to express the
alternating powers of the $56$-dimensional representation $V_7$ or the
$133$-dimensional representation $V_1$ in terms of the fundamental
modules of $E_7$ and their tensor products. We know that the character
$\chi_1$ of $V_1$ is $7 + \sum e^{\alpha}$, where the sum is over the
$126$ roots of $E_7$. Therefore we have $(-1) \eta_1 = \chi_1 -
7$. For the next example, we consider $\bigwedge^2 V_1 = V_3 \oplus
V_1$. This gives rise to the equation
$$
\eta_2 + 7 \cdot (-1)\eta_1 + {7 \choose 2} = \chi_3 + \chi_1
$$
which gives the relation $\eta_2 = \chi_3 - 6 \chi_1 + 28$. 

A similar analysis can be carried out to obtain all the other
identities used in our proofs, using the software \texttt{LiE},
available from \url{http://www-math.univ-poitiers.fr/\~maavl/LiE/}.

A more explicit method is to compute the expressions for the $\chi_i$
as Laurent polynomials in $s_1, \dots, s_6, r$ (note that $s_7 =
r^3/(s_1 \dots s_6)$), and then do the same for the $\eps_i$ or
$\eta_i$. The latter calculation is simplified by computing the power
sums $\sum (e^{\alpha})^i$ (for $\alpha$ running over the smallest
vectors of $E_7^*$ or $E_7$), for $1 \leq i \leq 7$ and then using
Newton's formulas to convert to the elementary symmetric polynomials,
which are $(-1)^i \eps_i$ or $(-1)^i \eta_i$. Finally, we check the
identities by direct computation in the Laurent polynomial ring (it
may be helpful to clear out denominators). This method has the
advantage that we obtain explicit expressions for the $\chi_i$ (and
then for $\lambda$ by Theorem \ref{explicitE7}) in terms of $s_1,
\dots, s_6, r$, which may then be used to generate examples such as
Example \ref{splitexampleE7}.

\section{Acknowledgements}

We thank David Vogan for very helpful comments regarding computations
in representation theory. The computer algebra systems
\texttt{gp/PARI}, \texttt{Magma}, \texttt{Maxima} and the software
\texttt{LiE} were used in the calculations for this paper. We thank
Matthias Sch\"utt and the referees for a careful reading of the paper
and for many helpful comments.

The auxiliary computer files for checking our calculations are
available from \url{http://arxiv.org/abs/1204.1531}. To access them,
download the source file for the paper. This will produce not only the
\LaTeX{} file for this paper, but also the computer code.  The file
\texttt{README.txt} gives an overview of the various computer files
involved.


\begin{thebibliography}{Sh1}
\bibitem[B]{B} N.~Bourbaki, \textit{Groupes et alg\`ebres de
    Lie. Chapitre IV: Groupes de Coxeter et syst\`emes de
    Tits. Chapitre V: Groupes engendr\'es par des
    r\'eflexions. Chapitre VI: syst\`emes de racines\/}, Actualit\'es
  Scientifiques et Industrielles, No.\ 1337 Hermann, Paris 1968.
\bibitem[C]{C} Carter, \textit{Lie algebras of finite and affine
    type\/}, Cambridge Studies in Advanced Mathematics, {\bf 96},
  Cambridge University Press, Cambridge, 2005.
\bibitem[ES]{ES} T.~Eguchi and K.~Sakai, \textit{Seiberg-Witten curve
    for E-string theory revisited\/}, Adv.\ Theor.\ Math.\ Phys.\ {\bf
    7} (2003), no.\ 3, 419--455.
\bibitem[JKZ]{JKZ} F.~Jouve, E.~Kowalski and D.~Zywina, \textit{An
  explicit integral polynomial whose splitting field has Galois group
  $W(E_8)$\/}, J.\ Th\'eor.\ Nombres Bordeaux {\bf 20} (2008), no.\ 3,
  761--782.
\bibitem[Ko]{Ko} K.~Kodaira, \textit{On compact analytic surfaces
    II-III}, Ann.\ of Math.\ {\bf 77} (1963), 563--626; {\bf 78}
  (1963), 1--40; Collected Works, III, 1269--1372, Iwanami and
  Princeton Univ.\ Press (1975).
\bibitem[Ku]{Ku} D.~S.~Kubert, \textit{Universal bounds on the torsion
  of elliptic curves\/}, Proc.\ London Math.\ Soc.\ (3) {\bf 33}
  (1976), no.\ 2, 193--237.
\bibitem[L]{L} M.~Lorenz, \textit{Multiplicative invariant theory\/},
  Encyclopaedia of Mathematical Sciences, {\bf 135}: Invariant Theory
  and Algebraic Transformation Groups, VI, Springer-Verlag, Berlin,
  2005.
\bibitem[OS]{OS} K.~Oguiso and T.~Shioda, \textit{The Mordell-Weil
  lattice of a rational elliptic surface\/},
  Comment.\ Math.\ Univ.\ St.\ Paul.\ {\bf 40} (1991), no.\ 1, 83--99.
\bibitem[Se]{Se} J.-P.~Serre, \textit{Lectures on the Mordell-Weil
    theorem. Translated from the French and edited by Martin Brown
    from notes by Michel Waldschmidt\/}, Aspects of Mathematics, {\bf
    E15}, Friedr.\ Vieweg \& Sohn, Braunschweig, 1989.
\bibitem[Sh1]{Sh1} T.~Shioda, \textit{On the Mordell-Weil lattices\/},
  Comment.\ Math.\ Univ.\ St.\ Paul.\ {\bf 39} (1990), no.\ 2,
  211--240.
\bibitem[Sh2]{Sh2} T.~Shioda, \textit{Construction of elliptic curves
  with high rank via the invariants of the Weyl groups\.},
  J.\ Math.\ Soc.\ Japan {\bf 43} (1991), no.\ 4, 673--719.
\bibitem[Sh3]{Sh3} T.~Shioda, \textit{Theory of Mordell-Weil
  lattices\/}, Proceedings of the International Congress of
  Mathematicians, Vol.\ I, II (Kyoto, 1990), 473--489,
  Math.\ Soc.\ Japan, Tokyo, 1991.
\bibitem[Sh4]{Sh4} T.~Shioda, \textit{Multiplicative excellent family
  of type $E_6$\/}, Proc.\ Japan Acad.\ Ser.\ A {\bf 88} (2012),
  no.\ 3, 46--51.
\bibitem[Sh5]{Sh5} T.~Shioda, \textit{A uniform construction of the
  root lattices $E_6, E_7, E_8$ and their dual lattices\/},
  Proc.\ Japan Acad.\ Ser.\ A Math.\ Sci.\ {\bf 71} (1995), no.\ 7,
  140--143.
\bibitem[Sh6]{Sh6} T.~Shioda, \textit{Gr\"obner basis, Mordell-Weil
  lattices and deformation of singularities. I.\/}, Proc.\ Japan
  Acad.\ Ser.\ A Math.\ Sci.\ {\bf 86} (2010), no.\ 2, 21--26.
\bibitem[Sh7]{Sh7} T.~Shioda, \textit{Gr\"obner basis, Mordell-Weil
  lattices and deformation of singularities. II.\/}, Proc.\ Japan
  Acad.\ Ser.\ A Math.\ Sci.\ {\bf 86} (2010), no.\ 2, 27--32.
\bibitem[Sh8]{Sh8} T.~Shioda, \textit{Some explicit integral
  polynomials with Galois group $W(E_8)$\/}, Proc.\ Japan
  Acad.\ Ser.\ A Math.\ Sci.\ {\bf 85} (2009), no.\ 8, 118--121.
\bibitem[SU]{SU} T.~Shioda and H.~Usui, \textit{Fundamental invariants
  of Weyl groups and excellent families of elliptic curves\/},
  Comment.\ Math.\ Univ.\ St.\ Paul.\ {\bf 41} (1992), no.\ 2,
  169--217.
\bibitem[T]{T} J.~Tate, \textit{Algorithm for determining the type of
    a singular fiber in an elliptic pencil}, Modular functions of one
  variable, IV (Proc.\ Internat.\ Summer School, Univ.\ Antwerp,
  Antwerp, 1972), pp. 33--52. Lecture Notes in Math., Vol.\ 476,
  Springer, Berlin, 1975.
\end{thebibliography}
\end{document}